\documentclass[11pt,oneside]{amsart}

\usepackage{graphics,color,pgf,comment}
\usepackage{epsfig}

 \usepackage[ansinew]{inputenc}
 \usepackage[all]{xy}
 \usepackage{hyperref}
\newdir{ >}{!/8pt/@{}*@{>}}

\usepackage{amssymb, amsmath,amsthm, mathtools}

\usepackage[margin=1.5in]{geometry}

\theoremstyle{plain}
\newtheorem{Teorema}{Theorem}[section]
\newtheorem{thm}[Teorema]{Theorem}
\newtheorem{cor}[Teorema]{Corollary}

\newtheorem{lemma}[Teorema]{Lemma}

\newtheorem{prop}[Teorema]{Proposition}

\newtheorem{conj}[Teorema]{Conjecture}
\newtheorem*{prop*}{Proposition}
\newtheorem{prop_intro}{Proposition}
\newtheorem{thm_intro}[prop_intro]{Theorem}

\newtheorem{cor_intro}[prop_intro]{Corollary}
\newtheorem{conj_intro}[prop_intro]{Conjecture}

\theoremstyle{definition}

\newtheorem*{defn_intro}{Definition}

\theoremstyle{remark}

\newtheorem{rem}[Teorema]{Remark}

\newcommand{\calT}{\ensuremath {\mathcal{T}}}
\newcommand{\calO} {\ensuremath {\mathcal{O}}}

\newcommand{\calL} {\ensuremath {\mathcal{L}}}

\newcommand{\calM} {\ensuremath {\mathcal{M}}}

\newcommand{\R} {\ensuremath {\mathbb{R}}}

\newcommand{\matH} {\ensuremath {\mathbb{H}}}

\DeclareMathOperator{\vol}{vol}

\title[On volumes of truncated tetrahedra]{On volumes of truncated tetrahedra \\ with constrained edge lengths}

\begin{document}

\author[]{R. Frigerio}
\address{Dipartimento di Matematica, Universit\`a di Pisa, Largo B. Pontecorvo 5, 56127 Pisa, Italy}
\email{frigerio@dm.unipi.it}

\author[]{M. Moraschini}
\address{Dipartimento di Matematica, Universit\`a di Pisa, Largo B. Pontecorvo 5, 56127 Pisa, Italy}
\email{moraschini@mail.dm.unipi.it}


\begin{abstract}
Truncated tetrahedra are the fundamental building blocks of hyperbolic $3$-manifolds with geodesic boundary. The study of their geometric properties
(in particular, of their volume)
has applications also in other areas of low-dimensional topology, like the computation of quantum invariants of $3$-manifolds and
the use of variational methods in the study of circle packings on surfaces.

The Lobachevsky--Schl{\"a}fli formula neatly describes the behaviour of the volume of truncated tetrahedra with respect to dihedral angles, while the dependence
of volume on edge lengths is worse understood. In this paper we prove that, for every $\ell<\ell_0$, where $\ell_0$ is an explicit constant, the regular truncated tetrahedron of edge length $\ell$ maximizes the volume 
among truncated tetrahedra whose edge lengths are all not smaller than $\ell$. 

This result provides a fundamental step in the computation of the ideal simplicial
volume of an infinite family of hyperbolic $3$-manifolds with geodesic boundary.
 \end{abstract}

\maketitle

 \section*{Introduction}

 The study of the geometry of hyperbolic truncated (also known as hyper-ideal) tetrahedra
 has applications to different areas of low-dimensional topology. For example, 
such tetrahedra play a fundamental role as building blocks of hyperbolic manifolds with geodesic boundary (see e.g.~\cite{FriPe,FriMaPe2,Kojima,Kojima2});
the relationship between volumes of truncated tetrahedra and  the  growth rate of $6j$-symbols provides a bridge
 between quantum invariants and hyperbolic volume, which is of interest in the context of the volume conjecture (see e.g.~\cite{Fra1, Fra2,costantinolast,chen});
 volumes of truncated tetrahedra also come into play when studying the geometry of circle packings on surfaces~\cite{schl1,schl2,spring}, as well as 
 the geometrization and some rigidity properties of hyperbolic cone manifolds~\cite{luoold,luonew}. 
  
 Compact truncated tetrahedra are parametrized by their dihedral angles in a very clean way: any $6$-tuple of positive dihedral
 angles is realized by
  a unique isometry class of truncated tetrahedra, provided that the sum of the angles assigned to any triple of edges emanating from a single vertex is smaller than $\pi$. 
  A truncated tetrahedron is \emph{regular} if any permutation of its vertices can be realized by an isometry of the tetrahedron or, equivalently, if the dihedral angles
  along its edges are all equal to each other. Regular truncated tetrahedra play an important role when studying hyperbolic manifolds with extremal volumes. For
  example, the smallest compact $3$-manifolds with geodesic boundary are precisely those manifolds which decompose into the union of two regular truncated tetrahedra with dihedral
  angles all equal to $\pi/6$~\cite{KM,M}. More in general,
    if $\calM_g$ is the set of compact orientable hyperbolic $3$-manifolds with geodesic boundary whose boundary is given  
  by a surface of genus $g$, then the elements of $\calM_g$ having the smallest volume 
  decompose into the union of $g$ copies of the regular truncated tetrahedron with dihedral angles equal to $\pi/(3g)$~\cite{M}.
  
  The behaviour of the volume of truncated tetrahedra with respect to dihedral angles is quite well understood. The Lobachevsky--Schl{\"a}fli formula (see Equation~\eqref{Lobachevsky--Schlafli-Formula} below) 
  implies that the volume decreases as angles increase, and also allows to prove that the volume is a concave function of dihedral angles~\cite{schlenker,schl1,schl2}. 
  Moreover, there exist closed formulas (involving Spence's dilogarithm) that explicitly compute the volume in terms of dihedral angles 
  (see e.g.~\cite{Ush}). 
  
  Since the very birth  of hyperbolic geometry, a lot of work  has been devoted to the study of the hyperbolic volume of (truncated) polytopes. 
  Indeed, a formula for the volume of geodesic simplices in hyperbolic $3$-space was already found by
J.~Bolyai, and  in 1836 Lobachevsky independently proved a nice volume formula for three-dimensional hyperbolic orthoschemes, as a function of three essential angles
(a $3$-dimensional orthoscheme is a geodesic simplex whose vertices admit an ordering $v_0,v_1,v_2,v_3$ such that the edge $v_0v_1$ is orthogonal to the face $v_0v_1v_2$ and the face $v_0v_1v_2$ is orthogonal to the edge $v_2v_3$). 
In 1852, Schl{\"a}fli computed the derivative of the volume with respect to dihedral angles for spherical simplices in any dimension. His result was then extended to the hyperbolic case by Kneser.
However, as pointed out e.g.~in~\cite{kellerhals3}, for hyperbolic $3$-dimensional simplices the Schl{\"a}fli formula was already known to Lobachevsky, and this is the reason why we call it 
\emph{Lobachevsky--Schl{\"a}fli formula} in this paper.

Much more recently, building on the work by Coxeter and by B\"ohm, Kellerhals  extended these investigations to deal with the so-called \emph{complete orthoschemes}, 
that are orthoschemes in which also ideal and/or truncated vertices are allowed
(see~\cite{kellerhals1, kellerhals2}). For $n\geq 5$, each simplex in an $n$-dimensional space of constant curvature is dissectable into orthoschemes~\cite{tsch}
(the question whether the same statement holds for every $n\geq 3$, which was 
conjectured by  Hadwiger in 1956~\cite{hadwiger}, is still open~\cite[Conjecture 23]{BKKS}).
Therefore, the volume 
of a (truncated) hyperbolic tetrahedron could  be computed in principle by exploiting Kellerhals' computation of the volumes of 
the orthoschemes into which the tetrahedron decomposes (however, this machinery usually leads to difficult computations and quite involved results).

An alternative approach to the computations of volumes for truncated polytopes in hyperbolic space is described in~\cite{molnar}, where the author develops
a projective metric theory which, for example, allows him to provide a new proof of Lobachevsky--Schl\"afli differential formula.

 In all the cited works, the volume is studied as a function of dihedral angles. In fact, the behaviour of volume with respect to edge lengths is a bit more mysterious. For example, the volume is neither concave nor convex as a function of edge lengths
  (see Remark~\ref{rem-conto-numerico}). 
  A formula for the volume of hyperbolic $3$-simplices in terms of their edge lengths is described in~\cite{ushi-length}. Moreover,
  in~\cite{hovath},
   Horv\'ath recently provided a new non-elementary integral for the volume of
orthoschemes (without exploiting the Lobachevsky--Schl\"afli differential formula), using edge lengths as the only parameters.  Even if both these papers do not explicitly
consider the case when proper vertices are replaced by truncation planes, it seems likely that the formulas described there could be extended to deal
with truncated tetrahedra. Nevertheless, applying this machinery to the problem we are interested in (see Theorem~\ref{main:thm}) does not seem straightforward.

Every (internal) edge of a truncated simplex of a decomposition of a hyperbolic manifold $M$ with geodesic boundary gives rise to a so-called \emph{orthogeodesic}, i.e.~a geodesic arc 
 intersecting the boundary of $M$ orthogonally at its endpoints. Lengths of orthogeodesics define the \emph{orthospectrum},  an interesting geometric object which has proved useful in the study
 of volumes
of hyperbolic manifolds with geodesic boundary (see~\cite{basma,bridge1,bridge2,calegari1,calegari2,bridge3,bridge4}). If $M$ is compact, then the orthospectrum of $M$ admits a positive minimum, say $\ell$,
hence no edge length of any  truncated tetrahedron appearing in any decomposition of $M$ can be smaller than $\ell$; moreover, the canonical Kojima decomposition of $M$ contains a truncated polyhedron
with one edge of length exactly equal to $\ell$~\cite{Kojima,Kojima2}.

These facts motivate our interest in the following class of truncated tetrahedra:

\begin{defn_intro}\label{Tdef}
 We denote by $\calT_\ell$ the set of 
isometry classes of truncated tetrahedra whose internal edge lengths are not smaller than $\ell$. We also denote by $\Delta_\ell\in \calT_\ell$ 
the (isometry class of the) regular truncated tetrahedron with all edge lengths equal to $\ell$. 
\end{defn_intro}

We will see in Section~\ref{truncated:sec} (where the reader can find the definition of internal edge)
that for every $\ell>0$ there exists a (unique, up to isometry) truncated tetrahedron with edge lengths all equal to $\ell$. Thus, the element
$\Delta_\ell\in\calT_\ell$ is indeed well defined.

As mentioned above, the smallest compact hyperbolic manifolds with geodesic boundary decompose into the union of two regular truncated tetrahedra with dihedral angles all equal
to $\pi/6$. An easy computation using formula~\eqref{anglestolengths} shows that the edge lengths of this tetrahedron are all equal to 
$$
\ell_0=\cosh^{-1}\left(\frac{3+\sqrt{3}}{4}\right)\ .
$$

The main result of this paper shows that, for every $\ell\leq \ell_0$, the tetrahedron $\Delta_\ell$ is the unique element of $\calT_\ell$ having the maximal volume:

\begin{thm_intro}\label{main:thm}
Let $\ell\leq \ell_0$ and
let $\Delta\in\calT_\ell$. Then
$$
\vol(\Delta)\leq\vol(\Delta_\ell)\ ,
$$
and
$$
\vol(\Delta)=\vol(\Delta_\ell)
$$
if and only if $\Delta=\Delta_\ell$.
\end{thm_intro}

It is known that regular truncated tetrahedra minimize the ratio between the volume and  the total area of truncation triangles (see Section~\ref{truncated:sec} for the definition
of truncation triangle): 
if $A_\partial(\Delta)$ denotes the sum of the areas of the truncation triangles of a truncated tetrahedron $\Delta$,
then for every $\ell>0$ the map
$$
\calT_\ell\to\R\, ,\qquad \Delta\mapsto \frac{\vol(\Delta)}{A_\partial (\Delta)}
$$
attains its unique global minimum at $\Delta_\ell$ (see \cite{M,Przeworski}, where a more general statement is proved, which holds in every dimension $\geq 3$).
Putting together this result with our Theorem~\ref{main:thm}  (and using that the area of a hyperbolic triangle is equal to $\pi$ minus the sum of its angles)
we obtain the following:

\begin{cor_intro}
Let $\ell\leq \ell_0$ and
let $\Delta\in\calT_\ell$. Then
$$
A_\partial(\Delta)\leq A_\partial(\Delta_\ell)\ ,
$$
and
$$
A_\partial(\Delta)=A_\partial(\Delta_\ell)
$$
if and only if $\Delta=\Delta_\ell$. Equivalently, the sum of the dihedral angles of $\Delta$ is not smaller than the sum of the dihedral angles of $\Delta_\ell$,
and the equality holds if and only if $\Delta=\Delta_\ell$.
\end{cor_intro}
In fact, it would be worth investigating further the relationship between our main result and the study of hyperball packings in hyperbolic space, for which we
refer the reader to~\cite{szirmai}
(see also~\cite{M,Przeworski}).

Theorem~\ref{main:thm} plays a fundamental role in the computation of the ideal simplicial volume of hyperbolic $3$-manifolds with geodesic boundary carried out in~\cite{FrMo2}.
We strongly believe that our main result should hold even without the assumption $\ell\leq \ell_0$. However, this technical restriction does not affect the applications described
in~\cite{FrMo2}, while allowing us to restrict our analysis of the volume function to particular acute-angled tetrahedra (see Proposition~\ref{Angoli-Acuti-et-al} below), whose geometry is easier to understand.

Indeed, here we formulate the following:

\begin{conj_intro}\label{main:conj}
 Let $\ell$ be any positive real number and
let $\Delta\in\calT_\ell$. Then
$$
\vol(\Delta)\leq\vol(\Delta_\ell)\ ,
$$
and
$$
\vol(\Delta)=\vol(\Delta_\ell)
$$
if and only if $\Delta=\Delta_\ell$.
\end{conj_intro}

We refer the reader to Section~\ref{final1} for the discussion of two related conjectures which, if solved in the affirmative, would imply Conjecture~\ref{main:conj}.

\smallskip 

Finally, we would like to mention that similar questions may arise also when considering hyperbolic tetrahedra having some proper vertices and some truncated
vertices. Such objects arise e.g.~when studying the so-called \emph{Lambert cube tilings}. For interesting results in this context we refer the reader to~\cite{kolp-murakami}, where
the authors compute the volume of doubly truncated hyperbolic tetrahedra (see also~\cite{kellerhals2} for related results).

\section{Truncated tetrahedra}\label{truncated:sec}
Let $P$ be
a tetrahedron and let $P^*$ be the combinatorial polyhedron obtained
by removing from $P$ small open stars of the vertices. We call
\emph{lateral hexagon} and \emph{truncation triangle} the intersection of
$P^*$ respectively with a face and with the link of a vertex
of $P$.
The edges of the truncation triangles are called \emph{boundary edges},
the other edges of $P^*$ are called \emph{internal edges}.
A \emph{truncated tetrahedron} 
is a realization of $P^*$ as
a compact polyhedron $\Delta\subseteq \matH^3$ in hyperbolic space,
such that the truncation triangles are geodesic triangles,
the lateral hexagons are geodesic hexagons, and truncation triangles and lateral
hexagons lie at right angles to each other (see Figure~\ref{truncated:fig}). Truncated tetrahedra are also known in the literature
with the names of generalized tetrahedra, or hyper-ideal tetrahedra (since they may be also defined as suitable truncations of 
genuine tetrahedra whose vertices are ``hyperideal'', meaning that they live in the complement of the closure of the projective model of hyperbolic space in
the real projective $3$-space). 

\begin{center}
 \begin{figure}
  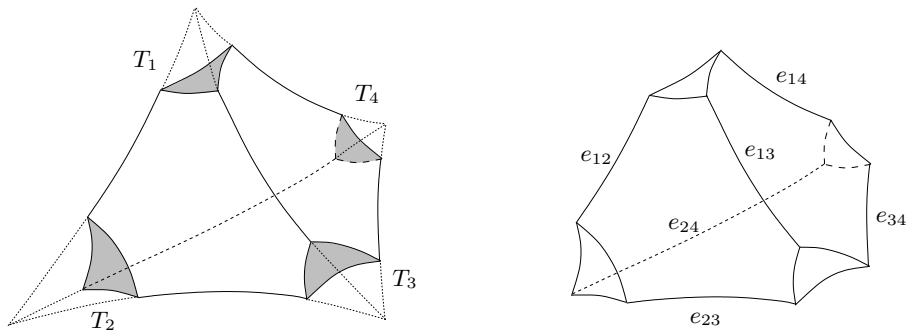
  \caption{A truncated tetrahedron. The ordering of truncation triangles on the left induces the labelling of internal edges on the right.}
  \label{truncated:fig}
 \end{figure}
\end{center}

\subsection{Parametrizing truncated tetrahedra}

A marking for a truncated tetrahedron is an ordering of its truncation triangles.
We denote by $\calT$ the space of isometry classes of marked truncated tetrahedra, i.e.~the set of equivalence classes of marked truncated tetrahedra,
where two tetrahedra are equivalent if they are isometric via an isometry which preserves the markings. If $\Delta$ is a marked truncated tetrahedron, we will denote
by $e_{ij}$, $1\leq i<j\leq 4$,  the internal edge with endpoints on the $i$-th and the $j$-th truncation triangle.

It is well-known that truncated tetrahedra are completely determined by their dihedral angles. 
 More precisely, for every $1\leq i<j\leq 4$,  let us define the function $\theta_{ij}\colon \calT\to \R$ such that $\theta_{ij}(\Delta)$ is the
 dihedral angle of $\Delta$ along $e_{ij}$, and set
 $$
 \Theta\colon \calT\to \R^6\, ,\qquad \Theta(\Delta)=(\theta_{12}(\Delta),\theta_{13}(\Delta),\theta_{14}(\Delta),\theta_{34}(\Delta),\theta_{24}(\Delta),\theta_{23}(\Delta))\ .
 $$
 If we denote by $\mathcal{O}$
 the convex open subset of ${\R}^6$ given by the $6$-tuples $(x_1,\ldots,x_6)$ of positive numbers that satisfy the system
 $$
\left\{
 \begin{array}{lcl}
 x_1+x_2+x_3&<&\pi\\
 x_1+x_5+x_6&<&\pi\\
 x_2+x_4+x_6&<&\pi\\
 x_3+x_4+x_5&<&\pi\ ,
\end{array}\right.
 $$
 then the map $\Theta$ establishes a bijection between $\calT$ and $\mathcal{O}$ (see e.g.~\cite{FriPe}). We can therefore endow $\calT$ with the structure of differentiable manifold
 for which the map $\Theta\colon \calT\to \calO$ is a diffeomorphism (hence, a chart).

 Truncated tetrahedra are determined also by their edge lengths, since dihedral angles may be expressed in terms of edge lengths (and viceversa) as follows.
For every $1\leq i<j\leq 4$,  let  $$\ell_{ij}\colon \calT\to \R$$ be the map which sends the (marked isometry class of) $\Delta$ to the hyperbolic length
of its edge $e_{ij}$.
Let us then set, for every $i\in\{1,2,3,4\}$, 
\begin{equation}\label{dtheta}
d_{i}\,=\,2\cos \theta_{ij} \cos\theta_{il} \cos\theta_{ik}+ \cos^2 \theta_{ij}
+\cos^2 \theta_{il}+\cos^2 \theta_{ik} -1\ ,
\end{equation}
\begin{equation}\label{zl}
z_i=2\cosh \ell_{jk}\cosh\ell_{kl}\cosh\ell_{lj}+\cosh^2\ell_{jk}+\cosh^2\ell_{kl}+\cosh^2\ell_{lj}-1\ , 
\end{equation}
where $j,k,l$ are such that $\{i,j,k,l\}=\{1,2,3,4\}$.
We then have the following equalities (see e.g.~\cite[Proposition 2.6]{FriPe} for the formula expressing lengths in terms of angles; the formula computing angles
in terms of lengths may be deduced in a very similar way):

\begin{equation}\label{anglestolengths}
\cosh \ell_{ij}={c_{ij}}\,\Big/\,{\sqrt{d_{i}
 d_{j}}}\ ,
\end{equation}
where
\begin{equation}\label{ctheta}
\begin{array}{rcl}
c_{ij}\!\!&=&\!\!\cos \theta_{ij} \left( \cos \theta_{il} \cos\theta_{jk}
        +\cos \theta_{ik} \cos\theta_{jl}\right)\\
        & & + \cos \theta_{il} \cos\theta_{jl}
        +\cos \theta_{ik} \cos\theta_{jk} + \cos\theta_{kl} \sin^2 \theta_{ij}\ ,
\end{array}
\end{equation}
and
\begin{equation}\label{lengthstoangles}
\cos \theta_{ij}={w_{ij}}\,\Big/\,{\sqrt{z_{k}
 z_{l}}}\ ,
\end{equation}
where
$$
\begin{array}{rcl}
w_{ij}\!\!&=&\!\!\cosh \ell_{ij} \left( \cosh \ell_{il} \cosh\ell_{jk}
        +\cosh \ell_{ik} \cosh\ell_{jl}\right)\\
        & & + \cosh \ell_{ik} \cosh\ell_{il}
        +\cosh \ell_{jk} \cosh\ell_{jl} - \sinh^2\ell_{ij} \cosh \ell_{kl}\ .
\end{array}\label{wl}
$$
As a consequence, if we set 
$$L\colon \calT\to \R^6\, ,\qquad L(\Delta)=(\ell_{12}(\Delta),\ell_{13}(\Delta),\ell_{14}(\Delta),
\ell_{34}(\Delta),\ell_{24}(\Delta),\ell_{23}(\Delta))\ ,
$$
and
$$
\mathcal{L}=L(\calT)\ ,
$$
then $\calL$ is an open subset of $\R^6$, and the map $L\colon \calT\to \calL$ is a diffeomorphism (hence, a chart for
$\calT$).

 
 \begin{lemma}[{\cite[Corollary 4.9]{luonew}}]\label{continuous:ext}
 The map $$ \Theta\circ L^{-1}\colon \mathcal{L}\to \mathcal{O}$$
 continuously extend to a function $\overline{\calL}\to\overline{\calO}$, 
where 
  $\overline{\calL}$ (resp.~$\overline\calO$) denotes the closure  of $\calL$ (resp.~of $\calO$) in $\R^6$.
 \end{lemma} 

\begin{rem}\label{flat:rem}
It is worth mentioning that the composition $L\circ\Theta^{-1}\colon \calO\to\calL$ does \emph{not} extend to a continuous function on $\calO$.
For example, when $\theta$ tends to $\pi/3$, the edge lengths of the regular truncated tetrahedron with dihedral angles all equal to $\theta$ diverge to $\infty$.
Even more delicate phenomena may occur. For example, let $W\subseteq \mathbb{H}^2$ be a right-angled octagon with cyclically ordered edges
$a_1,b_1,a_2,b_2,a_3,b_3,a_4,b_4$ and let $a_5\subseteq \mathbb{H}^2$ (resp.~$a_6\subseteq \mathbb{H}^2$) be the shortest geodesic segment joining 
$b_1$ and $b_3$ (resp.~$b_2$ and $b_4$) -- see Figure~\ref{octagon:fig}. If we embed $\mathbb{H}^2$ as a geodesic plane into $\mathbb{H}^3$,  
then it is not difficult to 
slightly deform $W$ into a 
truncated tetrahedron in $\mathbb{H}^3$, in such a way that the internal edges of the tetrahedron are given by the deformations of the $a_i$. 
As a consequence, if $\ell_i$ denotes the hyperbolic length of $a_i$, then the $6$-tuple $\overline{\ell}=(\ell_1,\ell_2,\ell_5,\ell_3,\ell_4,\ell_6)$ belongs to $\partial \calL$. 
Moreover, the continuous extension of $\Theta\circ L^{-1}$ maps any $6$-tuple $\overline{\ell}$ obtained in this way to $(0,0,\pi,0,0,\pi)$. As a consequence, no continuous extension of $L\circ\Theta^{-1}$
can be defined at $(0,0,\pi,0,0,\pi)$. 

For further results on the  funny behaviour of sequences of truncated tetrahedra degenerating into flat and/or non-compact polyhedra we refer the reader to~\cite{Miyamoto:degen}.
\end{rem}

\begin{center}
 \begin{figure}
  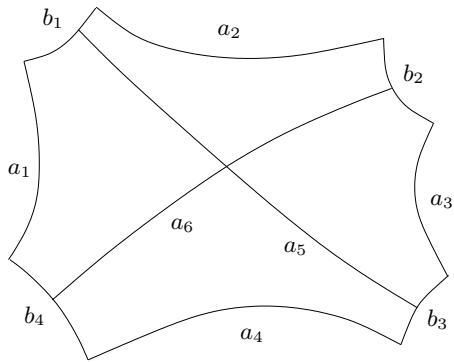
  \caption{A right-angled octagon may be seen as the limit point of a sequence of truncated tetrahedra that are becoming flat.}
  \label{octagon:fig}
 \end{figure}
\end{center}

\begin{rem}
The space $\calL$ of parameters describing truncated tetrahedra in terms of their edge lengths is more complicated than $\mathcal{O}$. For example, 
it is non-convex~\cite{luoold,luonew}. We refer the reader to~\cite[Proposition 4.5]{luonew} for a thorough description
of the spaces $\calL$ and $\partial\calL$.
\end{rem}

\section{The volume function}

We denote by 
$$
\vol\colon \calT\to \R
$$
the volume function, i.e.~the function which associates to any (marked isometry class of a) truncated tetrahedron its Riemannian volume. 
Recall that for every $\ell>0$ we denote by $\calT_\ell\subseteq \calT$ the space of truncated tetrahedra whose edge lengths are all not smaller than $\ell$, i.e.
$$
\calT_\ell=\{\Delta\in\calT\, |\, \ell_{ij}(\Delta)\geq \ell\ \textrm{for\ every}\ i,j\}\ .
$$
Our main theorem may be restated as follows: for every $\ell\leq \ell_0$, 
the function
 $$
 \vol\colon \calT_{\ell}\to \mathbb{R}
 $$
 has a unique global maximum, which is attained at the regular tetrahedron $\Delta_{\ell}$.

 Volumes of (truncated) hyperbolic tetrahedra satisfy the fundamental Lobachevsky--Schl{\"a}fli formula, which asserts that 
 \begin{equation}\label{Lobachevsky--Schlafli-Formula}
d{\vol} = - \frac{1}{2} \sum_{1\leq i<j\leq 4} \ell_{ij} d \theta_{ij}
\end{equation}
(see~\cite{Milnor:schla} for a proof of the formula for general polytopes in hyperbolic space and~\cite{schlenker,schl1,schl2} for a discussion of its application
to truncated tetrahedra). 

Henceforth, in order to avoid heavy notation we set $V_\Theta=\vol \circ \Theta^{-1}\colon \calO\to\R$ and
$V_L=\vol\circ L^{-1}\colon \calL\to\R$. Moreover, we simply denote by $\partial \vol/\partial \theta_{ij}$ the partial derivatives of $V_\Theta$, considered as functions
on $\calT$, i.e.~we set $\partial \vol/\partial \theta_{ij}=(\partial V_\Theta/\partial \theta_{ij})\circ \Theta$, and similarly
for $\partial \vol/\partial \ell_{ij}$. 

Let us point out three consequences of the 
Lobachevsky--Schl{\"a}fli differential formula (and of the concavity of the volume with respect to dihedral angles) that will prove useful later:

\begin{lemma}\label{decrease:increase}
Let $\Delta,\Delta'\in\calT$ be such that $\theta_{ij}(\Delta)\leq \theta_{ij}(\Delta')$ for every
$1\leq i<j\leq 4$. Then 
 $$
 \vol(\Delta)\geq \vol(\Delta')\ ,
 $$
 and the inequality is strict unless $\theta_{ij}(\Delta)= \theta_{ij}(\Delta')$ for every
$1\leq i<j\leq 4$.
\end{lemma}
\begin{proof}
For $1\leq i<j\leq 4$, let $\overline{\theta}_{ij}\colon [0,1]\to \mathbb{R}$ be defined by $\overline{\theta}_{ij}(t)=t\theta_{ij}(\Delta)+(1-t)\theta_{ij}(\Delta')$.
Since $\calO$ is convex, for every $t\in [0,1]$ we can consider the truncated tetrahedron $\Delta_t$ such that $\theta_{ij}(\Delta_t)=\overline{\theta}_{ij}(t)$
for every $1\leq i<j\leq 4$. Since $\theta_{ij}(\Delta)\leq \theta_{ij}(\Delta')$, 
the function $\overline{\theta}_{ij}$ is weakly decreasing, and it is strictly decreasing if $\theta_{ij}(\Delta)< \theta_{ij}(\Delta')$. Therefore,
from equation~\eqref{Lobachevsky--Schlafli-Formula} we obtain 
$$
\frac{d\vol(\Delta_t)}{dt}\geq 0
$$
for every $t\in [0,1]$, and the inequality is strict unless $\theta_{ij}(\Delta)= \theta_{ij}(\Delta')$ for every
$1\leq i<j\leq 4$. This implies that 
$$
 \vol(\Delta)=\vol(\Delta_1)\geq \vol(\Delta_0)= \vol(\Delta')\ ,
 $$
 and the inequality is strict unless $\theta_{ij}(\Delta)= \theta_{ij}(\Delta')$ for every
$1\leq i<j\leq 4$.
\end{proof}

\begin{lemma}\label{monotone}
The map
 $$
 \R\to\R\, ,\qquad \ell \mapsto \vol(\Delta_\ell)
 $$
 is strictly decreasing.
\end{lemma}
\begin{proof}
By using formula~\eqref{lengthstoangles} it is easily seen that 
the map $\ell\mapsto \theta_{ij}(\Delta_\ell)$ is strictly increasing for every $1\leq i<j\leq 4$, so the conclusion follows from Lemma~\ref{decrease:increase}.
\end{proof}

\begin{prop}\label{prop:volume:max:reg:somma:ang:uguale}
Let $0<\theta<2\pi$. Then
the maximum volume among the truncated tetrahedra in $$A_{\theta} = \left\{\Delta \in \mathcal{T}\ ,\ \sum_{1 \leq i<j \leq 4} \theta_{ij}(\Delta) = \theta\right\}$$ is attained at the regular truncated tetrahedron with dihedral angles all equal to $\theta/6$.
\end{prop}
\begin{proof}
Recall from~\cite{schlenker,schl1,schl2} that the function $V_\Theta\colon \calO\to \R$ is strictly concave. Therefore, since $\Theta(A_\theta)$ is a convex submanifold of $\calO$,
it suffices to show that $(\theta/6,\ldots,\theta/6)$ is the unique critical point of the restriction of $V_\Theta$ to $\Theta(A_\theta)$, or, equivalently, that the
unique critical point of the restriction of $\vol$ to $A_\theta$ is given by (the isometry class of) a regular tetrahedron.

Since $A_\theta$ is defined by the constraint $\sum_{1 \leq i < j\leq 4} \theta_{ij} = \theta$, a point $\Delta\in A_\theta$ is critical for the restriction
of $\vol$ to $A_\theta$ if and only if 
$$ \frac{\partial \vol}{\partial\theta_{ij}}(\Delta)=\frac{\partial \vol}{\partial\theta_{hk}}(\Delta)\quad \textrm{for\ every}\ 1\leq i<j\leq 4\, ,\ 1\leq h<k\leq 4\ .
$$
By the Lobachevsky--Schl{\"a}fli formula, this condition is equivalent to the fact that the egde lengths of $\Delta$ are all equal to each other, and this concludes the proof.
\end{proof}

\subsection{The behaviour of volume at $\partial \calL$}
It is known that, when considered as a function of dihedral angles, the volume continuously extends at points in $\partial\calO$:

\begin{thm}\cite{Rivin}\label{Rivin}
The function $$V_\Theta\colon \calO\to\R$$
continuously extends to $\overline{\calO}$.
\end{thm}

Rivin's Theorem allows us to control the volume on sequences of tetrahedra approaching the boundary of $\calL$:

\begin{prop}\label{boundarypoints}
The function $$V_L\colon \calL\to\R$$
continuously extends to a function $\overline{\calL}\to\R$, which we will still denote by $V_L$. Moreover, if $p\in \partial \calL\cap (0,+\infty)^6$, then 
$V_L(p)=0$.
\end{prop}
\begin{proof}
The first statement is an immediate consequence of Lemma~\ref{continuous:ext} and Theorem~\ref{Rivin}. By~\cite[Proposition 4.5]{luonew},
if $p\in\partial \calL\cap (0,+\infty)^6$ and $(\theta_{12},\ldots,\theta_{23})=(\Theta\circ L^{-1})(p)\in\overline{O}$,
then $\theta_{ij}=\theta_{hk}=\pi$, where $\{i,j,h,k\}=\{1,2,3,4\}$ (with a slight abuse, we denote by  $\Theta\circ L^{-1}$
also the continuous map extending $\Theta\circ L^{-1}$ to $\overline{\calL}$). Up to reordering the vertices, we may thus assume that
$(\Theta\circ L^{-1})(p)=(0,0,\pi,0,0,\pi)$. This $6$-tuple of angles may be realized as the limit point of a sequence of tetrahedra whose volume is tending to 0
(see
Remark~\ref{flat:rem}).  Since $V_L$ is continuous, this implies that $V_L(p)=0$.
\end{proof}

Our strategy is now very direct: if $\ell\leq \ell_0$ and $\Delta\in\calT_\ell$ is different from $\Delta_\ell$, we show that when decreasing the length of the longest edge of $\Delta$
the volume increases. Since $\Delta\neq \Delta_\ell$, the deformation just described may be performed inside $\calT_\ell$. With some care it is possible to prove that after a finite number of these volume-increasing deformations we  end up at the regular tetrahedron $\Delta_\ell$, and this concludes the proof. 
Let us state precisely our key result:

\begin{prop}\label{key}
Let $\ell\leq \ell_0$ and take an element $\Delta\in\calT_\ell$ such that $\vol(\Delta)\geq \vol(\Delta_\ell)$. Assume that 
$\ell_{12}(\Delta)\geq \ell_{ij}(\Delta)$ for every $1\leq i<j\leq 4$. Then
$$
\left(\frac{\partial \vol}{\partial \ell_{12}}\right)(\Delta)<0\ .
$$
\end{prop}

We now show how our main theorem may be deduced from Proposition~\ref{key}, whose proof is deferred to Section~\ref{sec:prop:key}. 
We fix a constant $0<\ell\leq \ell_0$, and for every $\Delta\in\calT_\ell$ we denote by $m(\Delta)\in\{1,\ldots,6\}$ the number of edges of $\Delta$ of maximal length.
We will show that $\vol(\Delta)\leq \vol(\Delta_\ell)$ (with equality only when $\Delta=\Delta_\ell$) by reverse induction on $m(\Delta)$. If $m(\Delta)=6$, then $\Delta$ is regular, and the conclusion follows from 
Lemma~\ref{monotone}. 

Suppose now $m(\Delta)<6$, and let us show that $\vol(\Delta)<\vol(\Delta_\ell)$.
We denote the biggest and the second biggest lengths of $\Delta$ by $\ell_a$ and $\ell_b$, respectively:
$$\ell_a=\max_{1\leq i<j\leq 4} \{\ell_{ij}(\Delta)\}\, ,\qquad \ell_b=\max \left(\{\ell_{ij}(\Delta)\, ,\ 1\leq i<j\leq 4\}\setminus \{\ell_a\}\right)\ .
$$
We consider the deformation of $\Delta$ corresponding to the following curve:
$$
\gamma\colon [0,\ell_a-\ell_b]\to (0,+\infty)^6\, ,\qquad \gamma(t)=(\ell_{12}(t),\ldots,\ell_{23}(t))\ ,
$$
where $\ell_{ij}(t)=\ell_a-t$ if $\ell_{ij}(\Delta)=\ell_a$ and $\ell_{ij}(t)=\ell_{ij}(\Delta)$ otherwise. Of course we may suppose that $\vol(\Delta)> \vol(\Delta_\ell)$
(otherwise we are done). Under this assumption we have the following:

\begin{lemma}\label{lemmino}
The curve $\gamma$ takes values in $L(\calT_\ell)$, and 
$V_L(\gamma(t_a-t_b))>\vol(\Delta)$.
\end{lemma}
\begin{proof}
By definition we have $L(\calT_\ell)=\calL\cap [\ell,+\infty)^6$. Since $\ell \leq \ell_b\leq \ell_a$, the fact that
$\gamma(t)\in [\ell,+\infty)^6$ is obvious, so
 we need to check that $\gamma([0,\ell_a-\ell_b])\subseteq \calL$, and  
$V_L(\gamma(t_a-t_b))>\vol(\Delta)$.

Let $t_0\in [0,\ell_a-\ell_b]$. We first show that, if $\gamma([0,t_0))\subseteq \calL$, then
$V_L(\gamma(t_0))>\vol(\Delta)$ (recall that $V_L$ is defined and continuous on the whole
of $\overline{\calL}$). In fact, up to reordering the vertices of $\Delta$,
we may suppose that $e_{12}$ is one of the edges of $\Delta$ of maximal length. For every $t\in[0,t_0)$,
if $\Delta_t=L^{-1}(\gamma(t))$, then
the set of the longest edges of $\Delta_t$ coincides with the set of the longest edges of $\Delta$.
In particular, Proposition~\ref{key} implies that  
for every $t\in [0,t_0)$ 
$$
(V_L\circ\gamma)'(t)=-m(\Delta)\cdot \left(\frac{\partial \vol}{\partial \ell_{12}}\right)(\Delta_t)>0\ .
$$
We thus have $V_L(\gamma(t_0))> V_L(\gamma(0))=\vol(\Delta)$, as claimed.

Suppose by contradiction that the lemma is false. Since $\calL$ is open and $V_L$ is continuous, what we have proved so far
implies that there exists $t_0\in [0,\ell_a-\ell_b]$ such that $\gamma([0,t_0))\subseteq \calL$ and
$\gamma(t_0)\in\partial\calL$. Moreover, $V_L(\gamma(t_0))>\vol(\Delta)>0$. 
Since $\gamma(t_0)$ also belongs to $\overline{\calT_\ell}\subseteq (0,+\infty)^6$, this
contradicts Proposition~\ref{boundarypoints}, thus concluding the proof of the lemma.
\end{proof}

Let us now set $\Delta'=L^{-1}(\gamma(\ell_a-\ell_b))$. By construction, $\Delta'$ belongs to $\calT_\ell$ and satisfies
$m(\Delta')>m(\Delta)$. Our inductive hypothesis now implies that $\vol(\Delta')\geq \vol(\Delta_\ell)$, while 
Lemma~\ref{lemmino} ensures that $\vol(\Delta)< \vol(\Delta')$. This concludes the proof of our main theorem, under the assumption that Proposition~\ref{key} holds.

\begin{rem}\label{rem-conto-numerico}
As already stated in the introduction, it is worth noting that the volume cannot be neither concave nor convex with respect to edge lengths. Indeed, one can numerically evaluate the Hessian matrix of the volume with respect to edge lengths in some point (e.g.~the regular tetrahedron with edge lengths all equal to $3$) and check that it has $2$ positive eigenvalues and $4$ negative eigenvalues. 

Notice that there is also a more explicit way to check that the volume cannot be concave with respect to edge lengths. Indeed, otherwise, being strictly decreasing on the line $t\mapsto L^{-1}(t,t,t,t,t,t)$, the volume of regular tetrahedra with big edge lengths would be negative.
\end{rem}

\section{Proof of Proposition~\ref{key}}\label{sec:prop:key}
Before going into the hearth of the proof of Proposition~\ref{key}, we establish some useful constraints on the dihedral angles of potential maxima
of the volume on $\calT_\ell$.
The assumption $\ell\leq\ell_0$ allows us to consider only tetrahedra with volume bigger than $\vol(\Delta_{\ell_0})$. Thanks to the following
proposition, this provides some restrictions on dihedral angles, that will prove useful in the computations we will carry out later. First recall (e.g.~from~\cite{M}) that
$$
\vol(\Delta_{\ell_0})=\left( 8\Lambda \left(\frac{\pi}4\right) - 3\int\limits_0^{\frac{\pi}{6}}
\cosh^{-1} \left(\frac{\cos t}{2\cos t -1}\right) \,{\rm d}t\right)\ \approx 3.226\ ,
$$
where $\Lambda\colon\R\to\R$ is the Lobachevsky function  defined by
$$ \Lambda(\theta)=-\int\limits_0^\theta \log |2 \sin u |\,\textrm{d}u.$$

 \begin{prop}\label{Angoli-Acuti-et-al}
 Let $\theta_{ij}$, $1\leq i<j\leq 4$ be the dihedral angles of a truncated tetrahedron $\Delta$ such that
 $$
 \vol(\Delta)\geq \vol(\Delta_{\ell_0})\ .
 $$
 Then:
 \begin{enumerate}
 \item $\theta_{12}+\theta_{13}+\theta_{14}+\theta_{23}+\theta_{24}+\theta_{34}\leq \pi$;
 \item $\theta_{ij}< \pi/2$ for every $i\neq j$;

\item if $i,j,k$ are distinct, then $\theta_{ij}+\theta_{ik}< (7/12)\pi$.
  \end{enumerate}
  \end{prop}
\begin{proof}
(1) By Proposition \ref{prop:volume:max:reg:somma:ang:uguale}, if $\theta=\theta_{12}+\theta_{13}+\theta_{14}+\theta_{23}+\theta_{24}+\theta_{34}$ and $\Delta'$ denotes
the regular tetrahedron with dihedral angles all equal to $\theta'=\theta/6$, then $\vol(\Delta')\geq \vol(\Delta)\geq \vol(\Delta_{\ell_0})$. Since the dihedral angles
of $\Delta_{\ell_0}$ are all equal to $\pi/6$,
by Lemma~\ref{decrease:increase} this implies  $\theta'\leq \pi/6$, whence the conclusion.

Items (2) and (3) can be proved by contradiction using Ushijima's formula for the volume of truncated tetrahedra~\cite{Ush}.
Let us first prove (2).
Assume e.g.~$\theta_{12} \geq \pi/2$. By Lemma \ref{decrease:increase} we have that 
$$\vol(\Delta) \leq \liminf_{\varepsilon \rightarrow 0^+} V_\Theta\left(\frac{\pi}{2}, \varepsilon, \varepsilon, \varepsilon, \varepsilon, \varepsilon\right)\ .$$ 
In order to compute the right-hand side of this inequality we then use~\cite[Thm. 1.1]{Ush}, which establishes an explicit formula for the volumes
of truncated tetrahedra in terms of their dihedral angles. Since this formula is continuous also on the closure of $\mathcal{O}$, we can simply evaluate it
at the point $(\pi/2,0,0,0,0,0)$. Before performing this computation, let us briefly recall how Ushijima's formula works.

One first consider the Gram matrix:
$$
G = \begin{pmatrix} 1 & -\cos\theta_{12} & -\cos\theta_{13} & -\cos\theta_{23} \\
-\cos\theta_{12} & 1 & -\cos\theta_{14} & -\cos\theta_{24} \\
-\cos\theta_{13} & -\cos\theta_{14} & 1 & -\cos\theta_{34} \\
-\cos\theta_{23} & -\cos\theta_{24} & -\cos\theta_{34} & 1
\end{pmatrix}
$$
and defines the following  complex numbers:
$$
z_1 = -2 \frac{\sin\theta_{12} \sin\theta_{34} + \sin\theta_{13} \sin\theta_{24} + \sin\theta_{14} \sin\theta_{23} - \sqrt{\det G }}{ad + be + cf + abf + ace + bcd + def + abcdef}\ ,
$$
$$
z_2 = -2 \frac{\sin\theta_{12} \sin\theta_{34} + \sin\theta_{13} \sin\theta_{24} + \sin\theta_{14} \sin\theta_{23} + \sqrt{\det G}}{ad + be + cf + abf + ace + bcd + def + abcdef}\ ,
$$
where
$$
a = \exp(i \theta_{12})\, ,\ b = \exp(i \theta_{13})\, ,\ c=\exp(i\theta_{14})\, ,$$
$$ d=\exp(i\theta_{34})\, ,\ e=\exp(i\theta_{24})\, ,\  f = \exp(i \theta_{23})\ .$$
Finally, the volume of the truncated tetrahedron having the $\theta_{ij}$ as dihedral angles is given by
$$\frac{1}{2} \Im(U(z_1) - U(z_2))\ ,$$
where $\Im$ denotes the imaginary part,
\begin{align*}
U(z_k) &= \frac{1}{2} \{Li_{2}(z_k) + Li_{2}(abde z_k) + Li_{2}(acdf z_k) + Li_{2}(bcef z_k) \\
&- Li_{2}(-abc z_k) - Li_{2}(-aef z_k) - Li_{2}(-bdf z_k) - Li_{2}(-cde z_k)\} \nonumber\ ,
\end{align*} 
and
$$
Li_2(z)=\sum_{k=1}^\infty \frac{z^k}{k^2}
$$
is Spence's dilogarithm.


In order to compute the volume of the tetrahedron with dihedral angles $(\pi/2,0,\ldots,0)$ we only need to observe that in this case
$\det G=-8$, while $a=i$ and $b=c=d=e=f=1$. Thus $$z_1=\frac{\sqrt{2}}{2}(1+ i) =-z_2$$ and 
\begin{align*}
\liminf_{\varepsilon \rightarrow 0^+} V_\Theta\left(\frac{\pi}{2}, \varepsilon, \varepsilon, \varepsilon, \varepsilon, \varepsilon\right) &=\frac{1}{2}\Im(U(z_1)-U(z_2)) \\
&= \frac{1}{2} \Im \Biggl( 2 Li_2\left(\frac{\sqrt{2}}{2}\left(1+ i\right) \right) - 2 Li_2\left(\frac{\sqrt{2}}{2}\left(1- i\right) \right) \\
&+ 2 Li_2\left(\frac{\sqrt{2}}{2}\left(-1+ i\right) \right) - 2 Li_2\left(\frac{\sqrt{2}}{2}\left(-1- i\right) \right) \Biggr) \\
&\approx 3.011 <  \vol(\Delta_{\ell_0})\ .
\end{align*}

Let us now prove (3).
Arguing as above, we may assume that $\theta_{12} + \theta_{13} \geq (7/12) \pi$, and we have
$$
\vol(\Delta) \leq \liminf_{\varepsilon \rightarrow 0^+} V_\Theta(\theta_{12}, \theta_{13}, \varepsilon, \varepsilon, \varepsilon, \varepsilon)\ .
$$
Let now $\theta'=(\theta_{12}+\theta_{13})/2\geq (7/24)\pi$. Thanks to the 
the symmetries of the tetrahedron and the concavity of the volume we have that the function $t\mapsto V_\Theta(\theta'+t,\theta'-t,\varepsilon,\ldots,\varepsilon)$ 
attains its maximum at $t=0$. Using Lemma~\ref{decrease:increase} we may thus conclude that 
$$
\vol(\Delta) \leq \liminf_{\varepsilon \rightarrow 0^+} V_\Theta(\theta',\theta', \varepsilon, \varepsilon, \varepsilon, \varepsilon)\leq 
\liminf_{\varepsilon \rightarrow 0^+} V_\Theta\left(\frac{7}{24}\pi, \frac{7}{24}\pi, \varepsilon, \varepsilon, \varepsilon, \varepsilon\right)\ .
$$
By applying Ushijima's formula to the $6$-tuple $((7/24)\pi,(7/24)\pi,0,0,0,0)$ we conclude that the right-hand side of this inequality is 
given by

$$\liminf_{\varepsilon \rightarrow 0^+} V_\Theta\left(\frac{7}{24}\pi, \frac{7}{24}\pi, \varepsilon, \varepsilon, \varepsilon, \varepsilon\right) \approx 3.210 < \vol(\Delta_{\ell_0})\ ,$$
and this concludes the proof.
\end{proof}

\subsection{The derivative of the volume with respect to lengths}
Of course, the main ingredient to compute the derivative of the volume with respect to lengths is Lobachevsky--Schl{\"a}fli formula~\eqref{Lobachevsky--Schlafli-Formula},
from which we get 
\begin{equation}\label{Lobachevsky--Schlafli-Formula2}
d \vol = - \frac{1}{2} \sum_{1\leq i<j\leq 4} \left(\sum_{1\leq k<l\leq 4} \ell_{kl} \frac{\partial \theta_{kl}}{\partial \ell_{ij}}\right) d\ell_{ij}\ . 
\end{equation}

We will first show how to exploit this formulation of the Lobachevsky--Schl{\"a}fli differential equality to reduce Proposition~\ref{key} to an explicit trigonometric inequality
over a specific subdomain of $\calO$ (see Proposition~\ref{tecnicofinale}). We will then provide 
a proof of this inequality, thus concluding the proof of Theorem~\ref{main:thm}.

Let us fix a length $\ell\leq \ell_0$, and take an element $\Delta\in\calT_{\ell}$. Also assume as in Proposition~\ref{key} that
$\vol(\Delta)\geq \vol(\Delta_\ell)$ and $\ell_{12}(\Delta)\geq \ell_{ij}(\Delta)$ for every $1\leq i<j\leq 4$.
We need to show that
\begin{equation*}
\left(\frac{\partial \vol}{\partial \ell_{12}}\right) (\Delta)<0\ .
\end{equation*}

Let us denote by $\theta_{ij}$, $1\leq i<j\leq 4$, the dihedral angles of $\Delta$. 
Putting together the computation of the partial derivatives $\partial \theta_{ij}/\partial \ell_{hk}$ provided by the main theorem of~\cite{Guo} and
the Lobachevsky--Schl{\"a}fli formula~\eqref{Lobachevsky--Schlafli-Formula2} we obtain
\begin{align*}
\frac{\partial \vol}{\partial \ell_{12}}(\Delta)&=-\frac{1}{2} \sum_{1\leq i<j\leq 4} \ell_{ij}\frac{\partial \theta_{ij}}{\partial \ell_{12}}\\ &=
k\big(\ell_{12}(\cos\theta_{12}(\cos\theta_{13}\cos\theta_{23}+\cos\theta_{14}\cos\theta_{24})\\ &+\cos\theta_{13}\cos\theta_{24}+\cos\theta_{14}\cos\theta_{23})
\\ &-\ell_{13}\sin\theta_{12}\sin\theta_{13}\cos\theta_{23}
-\ell_{14}\sin\theta_{12}\sin\theta_{14}\cos\theta_{24} +\ell_{34}\sin\theta_{12}\sin\theta_{34}
\\ &-\ell_{24}\sin\theta_{12}\sin\theta_{24}\cos\theta_{14}
-\ell_{23}\sin\theta_{12}\sin\theta_{23}\cos\theta_{13}\big)\ ,
\end{align*}
where $k$ is a negative constant. Therefore, in order to conclude we need to prove that
\begin{align*}
&
\ell_{12}(\cos\theta_{12}(\cos\theta_{13}\cos\theta_{23}+\cos\theta_{14}\cos\theta_{24})+\cos\theta_{13}\cos\theta_{24}+\cos\theta_{14}\cos\theta_{23})\\ + &\ell_{34}\sin\theta_{12}\sin\theta_{34}\\ > & 
\ell_{13}\sin\theta_{12}\sin\theta_{13}\cos\theta_{23}+\ell_{14}\sin\theta_{12}\sin\theta_{14}\cos\theta_{24} +\ell_{24}\sin\theta_{12}\sin\theta_{24}\cos\theta_{14}
\\ + & \ell_{23}\sin\theta_{12}\sin\theta_{23}\cos\theta_{13}\ .
\end{align*}
Since 
$\vol(\Delta)\geq \vol(\Delta_{\ell})\geq \vol(\Delta_{{\ell_0}})$, Proposition~\ref{Angoli-Acuti-et-al} implies
  that $0<\theta_{ij}<\pi/2$ for every $1\leq i<j\leq 4$. 
Therefore, $\sin\theta_{ij}>0$ and $\cos\theta_{ij}>0$ for every $1\leq i<j\leq 4$.
Since $\ell_{34}\sin\theta_{12}\sin\theta_{34}>0$ and $\ell_{12}\geq \ell_{ij}$ for every $i,j\in\{1,\ldots,4\}$, it is then sufficient to show that
\begin{align*}
&
\cos\theta_{12}(\cos\theta_{13}\cos\theta_{23}+\cos\theta_{14}\cos\theta_{24})+\cos\theta_{13}\cos\theta_{24}+\cos\theta_{14}\cos\theta_{23}\\ \geq &
\sin\theta_{12}\sin\theta_{13}\cos\theta_{23}+\sin\theta_{12}\sin\theta_{14}\cos\theta_{24} +\sin\theta_{12}\sin\theta_{24}\cos\theta_{14}
\\  & + \sin\theta_{12}\sin\theta_{23}\cos\theta_{13}\\
= & \sin\theta_{12}\left(\sin\left(\theta_{13}+\theta_{23}\right)+\sin\left(\theta_{14}+\theta_{24}\right)\right)\  .
\end{align*}

We have thus reduced  Proposition~\ref{key} to  the following:

\begin{prop}\label{tecnicofinale}
Let $\theta_{ij}$, $i,j\in\{1,\ldots,4\}$, be the dihedral angles of a truncated tetrahedron $\Delta$ such that
$\vol(\Delta)\geq \vol(\Delta_{\ell_0})$. Then
\begin{align*}
& \cos\theta_{12}(\cos\theta_{13}\cos\theta_{23}+\cos\theta_{14}\cos\theta_{24})+\cos\theta_{13}\cos\theta_{24}+\cos\theta_{14}\cos\theta_{23}\\ \geq &
\sin\theta_{12}(\sin(\theta_{13}+\theta_{23})+\sin(\theta_{14}+\theta_{24}) )
\ . 
\end{align*}
\end{prop}

The rest of this section is entirely devoted to the proof of Proposition~\ref{tecnicofinale}. We first establish some preliminary lemmas.


\begin{lemma}\label{nuovestime2}
 Let $\theta_{ij}$, $1\leq i<j\leq 4$, be the dihedral angles of a truncated tetrahedron $\Delta$ such that $\pi/6\leq \theta_{12}\leq \pi/3$ and
 $
 \vol(\Delta)\geq \vol(\Delta_{\ell_0})
 $.
 Then
  $$
\cos\theta_{13}\cos\theta_{24}+\cos\theta_{14}\cos\theta_{23}  \geq 2\sin\frac{\theta_{12}}{2}\ .
$$
   \end{lemma}
\begin{proof}
Since $\theta_{13}+\theta_{24}+\theta_{14}+\theta_{23}\leq \pi-\theta_{12}$, without loss of generality we may assume that
$\theta_{13}+\theta_{24}\leq (\pi-\theta_{12})/2$. 

Suppose now $\theta_{13}+\theta_{24}\leq \pi/6$. Then $\cos\theta_{13}\cos\theta_{24}\geq \cos \pi/6=\sqrt{3}/2$. Moreover,
 $\theta_{14}$ and $\theta_{23}$ are both strictly smaller than $(7/12)\pi-\theta_{12}$ (which is smaller than $(7/12)\pi-\pi/6<\pi/2$), hence
\begin{align*}
 \cos\theta_{13}\cos\theta_{24}+\cos\theta_{14}\cos\theta_{23}& \geq \frac{\sqrt{3}}{2}+ \cos^2((7/12)\pi-\theta_{12})\\  & =  
\frac{\sqrt{3}+1+\cos((7/6)\pi-2\theta_{12})}{2}
\\ & \geq 2\sin \frac{\theta_{12}}{2}\ ,
\end{align*}
where the last inequality is due to the fact that the function
$$
\left(\frac{\pi}{6},\frac{\pi}{3}\right) \to \R\, ,\qquad
\theta \mapsto \frac{\sqrt{3}+1+\cos((7/6)\pi-2\theta)}{2}-2\sin\frac{\theta}{2}
$$
attains it minimum at $\theta=(4/15)\pi$, where it is positive.
This concludes the proof under the assumption
$\theta_{13}+\theta_{24}\leq \pi/6$.

We may thus assume that $\pi/6\leq \theta_{13}+\theta_{24}\leq (\pi-\theta_{12})/2$. Since the map 
$x\mapsto \cos x-\cos (x+\theta_{12})$ is increasing on $(0,(\pi-\theta_{12})/2)$, we then have
$$
\cos (\theta_{13}+\theta_{24})-\cos (\theta_{13}+\theta_{24}+\theta_{12})\geq \cos (\pi/6)-\cos(\pi/6-\theta_{12})\ .
$$
Also observe that, since $\max\{\theta_{13},\theta_{23},\theta_{14},\theta_{24}\}\leq (7/12)\pi-\theta_{12}$, we have
$$
|\theta_{13}-\theta_{24}|\leq \frac{7\pi}{12}-\theta_{12}\, ,\quad
|\theta_{14}-\theta_{23}|\leq \frac{7\pi}{12}-\theta_{12}\, ,
$$
while the condition
$\max\{\theta_{13},\theta_{24}\}+\max\{\theta_{23},\theta_{14}\}\leq (7/12)\pi$ implies that
$$
|\theta_{13}-\theta_{24}|+|\theta_{14}-\theta_{23}|\leq \frac{7\pi}{12}\ ,
$$
hence
\begin{align*}
\cos (\theta_{13}-\theta_{24})+\cos (\theta_{14}-\theta_{23}) &=\cos |\theta_{13}-\theta_{24}|+\cos |\theta_{14}-\theta_{23}| \\
&\geq \cos \left(\frac{7\pi}{12}-\theta_{12}\right)+\cos \theta_{12}\ .
\end{align*}
Putting together all these inequalities we get
\begin{align*}
& \cos\theta_{13}\cos\theta_{24}+\cos\theta_{14}\cos\theta_{23} \\ = & \frac{\cos (\theta_{13}+\theta_{24})+\cos (\theta_{14}+\theta_{23})}{2}+
 \frac{\cos (\theta_{13}-\theta_{24})+\cos (\theta_{14}-\theta_{23})}{2}\\ \geq &
 \frac{\cos (\theta_{13}+\theta_{24})+\cos (\pi-\theta_{12}-\theta_{13}-\theta_{24})}{2}+
 \frac{\cos (\theta_{13}-\theta_{24})+\cos (\theta_{14}-\theta_{23})}{2}\\
 = & \frac{\cos (\theta_{13}+\theta_{24})-\cos (\theta_{13}+\theta_{24}+\theta_{12})}{2}+
 \frac{\cos (\theta_{13}-\theta_{24})+\cos (\theta_{14}-\theta_{23})}{2}\\
 \geq &
  \frac{\cos (\pi/6)-\cos(\pi/6+\theta_{12})+ \cos \left((7/12)\pi-\theta_{12}\right)+\cos \theta_{12}}{2}\\
  \geq & 2\sin\frac{\theta_{12}}{2}\ ,
\end{align*}
where the last inequality is due to the fact that
 the function $$\theta\mapsto  \frac{\cos (\pi/6)-\cos(\pi/6+\theta)+ \cos \left((7/12)\pi-\theta\right)+\cos \theta}{2}-2\sin\frac{\theta}{2}$$ 
is decreasing on $(\pi/6,\pi/3)$ and positive at $\pi/3$.

\end{proof}

\begin{lemma}\label{nuovestime}
 Let $\theta_{ij}$, $1\leq i<j\leq 4$ be the dihedral angles of a truncated tetrahedron $\Delta$ such that
 $
 \vol(\Delta)\geq \vol(\Delta_{\ell_0})
 $.
 Then
 $$
 \cos\theta_{13}\cos\theta_{24}+\cos\theta_{14}\cos\theta_{23}\geq 1-\sin \frac{\pi}{12}\ .
 $$
   \end{lemma}
\begin{proof}
Without loss of generality we may assume $\theta_{13}=\max\{\theta_{13},\theta_{14},\theta_{23},\theta_{24}\}$.
If $\theta_{13}\leq (7/24)\pi$, then 
$$ \cos \theta_{13}\cos\theta_{24}+\cos\theta_{14}\cos\theta_{23}\geq
2\cos^2\frac{7\pi}{24}=1+\cos\frac{7\pi}{12}=1-\sin\frac{\pi}{12}\ .
$$
Otherwise, both $\theta_{23}$ and $\theta_{14}$ are not bigger than $(7/12)\pi-\theta_{13}$, and
$\theta_{24}\leq \theta_{13}$, so 
\begin{align*}
&\cos \theta_{13}\cos\theta_{24}+\cos\theta_{14}\cos\theta_{23}\geq 
\cos^2\theta_{13}+\cos^2\left(\frac{7\pi}{12}-\theta_{13}\right)\\ = &
1+\frac{\cos 2\theta_{13}+\cos (7\pi/6-2\theta_{13})}{2}\geq 1+\cos\frac{7\pi}{12}=1-\sin\frac{\pi}{12}\ ,
\end{align*}
since the map
$\theta\mapsto \cos 2\theta+\cos (7\pi/6-2\theta)$ is increasing on $(7\pi/24,\pi/2)$.
\end{proof}

\begin{lemma}\label{minimoserve}
Let $\theta_{ij}$, $1\leq i<j\leq 4$ be the dihedral angles of a truncated tetrahedron $\Delta$ such that
 $
 \vol(\Delta)\geq \vol(\Delta_{\ell_0})
 $.
 Then
\begin{align*}
&\cos\theta_{12}\left(\cos\theta_{13}\cos\theta_{23}+\cos\theta_{14}\cos\theta_{24}\right)-\sin \theta_{12}(\sin(\theta_{13}+\theta_{23})+\sin(\theta_{14}+\theta_{24})) \\
&\geq -2\sin\frac{\theta_{12}}{2}\ .
\end{align*}
\end{lemma}
\begin{proof}
Let $k=\theta_{13}+\theta_{23}$, $h=\theta_{14}+\theta_{24}$. 
Using that $\theta_{13},\theta_{23},\theta_{14}$ and $\theta_{24}$ belong to $(0,\pi/2)$ it is easy to check that $\cos\theta_{13}\cos\theta_{23}\geq \max \{0,\cos k\}$ and
$\cos\theta_{14}\cos\theta_{24}\geq \max \{0,\cos h\}$. Assume first that $h\leq \pi/2$ and $k\leq \pi/2$. Then
\begin{align*}
&\cos\theta_{12}\left(\cos\theta_{13}\cos\theta_{23}+\cos\theta_{14}\cos\theta_{24}\right)-\sin \theta_{12}(\sin(\theta_{13}+\theta_{23})+\sin(\theta_{14}+\theta_{24}))\\ \geq &
\cos\theta_{12}\cos h-\sin\theta_{12}\sin h+\cos\theta_{12}\cos k-\sin\theta_{12}\sin k\\ & =\cos(\theta_{12}+h)+\cos(\theta_{12}+k) 
=  2\cos \left(\frac{h+k}{2}+\theta_{12}\right)\cos\frac{h-k}{2}\ .
\end{align*}
Since $h+k\leq \pi-\theta_{12}$ we have $0\leq (h+k)/2+\theta_{12}\leq \pi/2+\theta_{12}/2$, so $\cos((h+k)/2+\theta_{12})\geq \cos(\pi/2+\theta_{12}/2)=-\sin(\theta_{12}/2)$,
and the conclusion follows.

Assume now $h\geq \pi/2$. Since $h+k\leq \pi-\theta_{12}$, we necessarily have $k\leq \pi/2-\theta_{12}$, hence
 \begin{align*}
&\cos\theta_{12}\left(\cos\theta_{13}\cos\theta_{23}+\cos\theta_{14}\cos\theta_{24}\right)-\sin \theta_{12}(\sin(\theta_{13}+\theta_{23})+\sin(\theta_{14}+\theta_{24}))\\ \geq &
-\sin\theta_{12}\sin h+\cos\theta_{12}\cos k-\sin\theta_{12}\sin k\geq -\sin\theta_{12} +\cos(\theta_{12}+k)\\ 
\geq & -\sin\theta_{12}\geq -2\sin\frac{\theta_{12}}{2}\ ,
\end{align*}
and we are done.
\end{proof}

We are finally ready to prove Proposition~\ref{tecnicofinale}.
Let $\theta_{ij}$, $i,j\in\{1,\ldots,4\}$, $i<j$ be the dihedral angles of a truncated tetrahedron $\Delta$ such that
$\vol(\Delta)\geq \vol(\Delta_{\ell_0})$.

By Lemma~\ref{minimoserve}, it is sufficient to prove that
\begin{equation}\label{stimariscritta}
\cos\theta_{13}\cos\theta_{24}+\cos\theta_{14}\cos\theta_{23} \geq 
2\sin\frac{\theta_{12}}{2}\ .
\end{equation}
If $\theta_{12}\leq \pi/6$, then $2\sin (\theta_{12}/2)\leq 2\sin (\pi/12)$, so 
this inequality readily follows from Lemma~\ref{nuovestime}, together with the fact that
$3\sin(\pi/12)<1$. 

The case when 
$\pi/6\leq \theta_{12}\leq \pi/3$ is proved in Lemma~\ref{nuovestime2}, so we may assume $\pi/3<\theta_{12}<\pi/2$.

Since $\max\{\theta_{13},\theta_{14},\theta_{23},\theta_{24}\}\leq (7/12)\pi-\theta_{12}$, we have
$$
\cos\theta_{13}\cos\theta_{24}+\cos\theta_{14}\cos\theta_{23}\geq 2\cos^2((7/12)\pi-\theta_{12})=1+\cos \left(\frac{7\pi}{6}-2\theta_{12}\right)\ ,
$$
and the conclusion follows since the map
$$
\theta\mapsto 1+\cos \left(\frac{7\pi}{6}-2\theta\right)-2\sin \frac{\theta}{2}
$$
vanishes at $\theta=\pi/3$, and it is strictly increasing on $[\pi/3,\pi/2)$.

\section{Concluding remarks}\label{final1}
As mentioned in the introduction, we believe that Theorem~\ref{main:thm} holds in general, i.e.~without the assumption $\ell\leq \ell_0$. 
It seems likely that a useful tool to approach Conjecture~\ref{main:conj} should be the concavity of the volume as a function of dihedral angles.
As already observed in Remark~\ref{rem-conto-numerico}, the volume  cannot be neither concave nor convex with respect to edge lengths.
Moreover, 
the domain $\calL\subseteq \R^6$ is not convex. These facts seem to suggest that
one cannot completely avoid the use of dihedral angles when trying to maximize or minimize volumes of truncated tetrahedra, even when constraints are expressed in terms of edge lengths.

\begin{conj}\label{prima:conj}
Take $\ell\in\R$, 
 let $\Delta\in\calT_{\ell}$ and denote by $\theta_{ij}$, $1\leq i<j\leq 4$, the dihedral angles of $\Delta$. Also
 let $\theta'=(\theta_{12}+\theta_{13}+\theta_{14}+\theta_{23}+\theta_{24}+\theta_{34})/6$, and denote by $\Delta'\in\calT$ the
 regular tetrahedron with dihedral angles all equal to $\theta'$. Then $\Delta'\in \calT_{\ell}$. 
 
 In other words, if every edge length of $\Delta$ is not smaller than $\ell$ and $\Delta'$ is obtained
 by assigning to every edge the average of the dihedral angles of $\Delta$, then the edge length of $\Delta'$ is not strictly smaller than $\ell$.
 \end{conj}
 
 Conjecture~\ref{prima:conj} easily implies Conjecture~\ref{main:conj}: indeed, if Conjecture~\ref{prima:conj} holds, then Proposition~\ref{prop:volume:max:reg:somma:ang:uguale} implies
 that
 $$\vol(\Delta)\leq \vol(\Delta')\leq \vol(\Delta_{\ell})\ ,$$
 where the last inequality is due to Lemma~\ref{monotone}.
 
 In order to state another related conjecture, let us first observe that the symmetric group $\mathfrak{S}_4$ acts on $\calT$ via permutations of the (removed) vertices
 of truncated tetrahedra:
 if $\sigma\in \mathfrak{S}_4$, then $\theta_{ij}(\sigma\cdot\Delta)=\theta_{\sigma(i)\sigma(j)}(\Delta)$ for every $1\leq i<j\leq 4$. 
 The tetrahedra $\Delta$ and $\sigma\cdot \Delta$ are isometric (via an isometry which does not preserve the marking), hence
 $\vol(\sigma\cdot \Delta)=\vol(\Delta)$, and $\sigma\cdot \Delta\in \calT_{\ell}$ if $\Delta\in\calT_{\ell}$. 
 It readily follows from the definitions that $\Delta$ is regular if and only if $\sigma\cdot\Delta=\Delta$
 for every $\sigma\in\mathfrak{S}_4$. 
 
 \begin{conj}\label{prima2:conj}
Take $\ell\in\R$, 
 let $\Delta\in\calT_{\ell}$ be a non-regular truncated tetrahedron and denote by $\calO_\Delta\subseteq \calO$ the convex hull of the set
 $$
 \{\Theta(\sigma\cdot \Delta)\, ,\ \sigma\in\mathfrak{S}_4\}\ .
 $$
 Then the intersection
 $$
\left( \calO_\Delta\setminus \{\Theta(\sigma\cdot \Delta)\, ,\ \sigma\in\mathfrak{S}_4\}\right)\cap \Theta(\calT_{\ell})
 $$
 is non-empty.

 In other words, the convex hull (with respect to dihedral angles) of the tetrahedra obtained as symmetric images of $\Delta$ intersects
 $\calT_{\ell}$ in other points besides its vertices.
 \end{conj}
 
 Of course, Conjecture~\ref{prima:conj} implies Conjecture~\ref{prima2:conj}. An easy argument exploiting again the concavity of the volume shows that
 Conjecture~\ref{prima2:conj} still implies
 Conjecture~\ref{main:conj}, provided that the volume attains a maximum on the non-compact space $\calT_{\ell}$
 (which
can be probably proved  without much effort).

\bibliographystyle{amsalpha}
\bibliography{bibliovolume}

\end{document}